\documentclass[10pt,a4paper]{article}

\usepackage[latin1]{inputenc}
\usepackage[left=2.00cm, right=2.00cm, top=2.00cm]{geometry}
\usepackage{amsmath}
\usepackage{amsthm}
\usepackage{amsfonts}
\usepackage{amssymb}
\usepackage{graphicx}
\usepackage{tikz-cd}
\usepackage{subcaption}
\usepackage{tikz}
\usepackage{extpfeil} 
\usetikzlibrary{arrows}
\usetikzlibrary{arrows.meta}  
\usepackage[section]{placeins}

\theoremstyle{plain}
\newtheorem{thm}{Theorem}[section]
\newtheorem{cor}[thm]{Corollary}
\newtheorem{propn}[thm]{Proposition}
\newtheorem{lem}[thm]{Lemma}
\theoremstyle{definition}
\newtheorem{ex}{Example}[section]
\numberwithin{equation}{section}

\DeclareMathOperator{\Imag}{Im}

\DeclareMathOperator{\Ad}{Ad}

\DeclareMathOperator{\Tr}{Trace}
\DeclareMathOperator{\aff}{Aff}
\DeclareMathOperator{\affo}{\mathfrak{aff}}

\newcommand{\orb}{\mathcal{O}}

\title{A Bijection Between the Adjoint and Coadjoint Orbits of a Semidirect Product}
\author{Philip Arathoon}
\date{June 2018}
\begin{document}
	\maketitle
	 \begin{abstract}
	 	We prove that there exists a geometric bijection between the sets of adjoint and coadjoint orbits of a semidirect product, provided a similar bijection holds for particular subgroups. We also show that under certain conditions the homotopy types of any two orbits in bijection with each other are the same. We apply our theory to the examples of the affine group and the Poincar\'{e} group, and discuss the limitations and extent of this result to other groups.
	 \end{abstract}

	\section*{Background and Outline}
	For a reductive Lie algebra the adjoint and coadjoint representations are isomorphic; consequently, the orbits are identical. In \cite{cush77} a method is devised to obtain normal forms for the adjoint orbits of any semisimple Lie algebra, real or complex. These methods are then extended in \cite{cush06} and applied to the Poincar\'{e} group. The Poincar\'{e} group is an example of a non-reductive group, and consequently there is no reason, in general, to expect any relation between the adjoint and coadjoint representations. Nevertheless, in \cite{cush06} a ``curious bijection'' is found between the normal forms of both representations. 
	
	Before proceeding any further with the details, it might be pertinent to exhibit a hands-on example of what we mean by such a bijection. In Figure~\ref{aff1} we illustrate the orbits of the group \(\aff(1)\) of affine isomorphisms of the real line. The adjoint and coadjoint representations are not isomorphic, indeed, the orbits are different; and yet, there seems to be some sort of bijection between the two. In our work in \cite{paper1} we explore in detail the orbits of the Euclidean group and prove such a bijection result. However, for other semidirect products, such as the Poincar\'{e} group, our methods no longer apply. Thus the purpose of this paper is to prove a bijection result for a wider class of semidirect products.

The study of coadjoint orbits, particularly those of a semidirect product, is a large and venerable subject; and one which we will mostly be able to sidestep. For a greater insight into the physical significance and applications of this study, consult \cite{baguis,bigstages,particles,semistages,stern}, to name but a few.

In the first section we prove our central result: that their exists a geometric bijection between the sets of adjoint and coadjoint orbits of a semidirect product provided a similar bijection holds for particular subgroups. To be precise, for a semidirect product \(G=H\ltimes V\) the particular subgroups in question are: the Wigner little groups \(H_p\subset H\) which are the stabilisers of a vector \(p\in V^*\); and the centralizer subgroups \(H_\omega\subset H\) which are the stabilisers of an element \(\omega\) in the Lie algebra of \(H\). We prove that the bijection result holds if: there is a bijection between the sets of adjoint and coadjoint orbits of the groups \(H_p\); and a bijection between the sets of orbits of \(H_\omega\) with respect to a particular representation and its contragredient. Thus, the task of establishing an orbit bijection is reduced to a similar task, albeit for a collection of `smaller' groups. 

In the second and third sections we demonstrate this bijection for the examples of the affine group of isomorphims of affine space, and the Poincar\'{e} group of affine linear maps preserving Minkowski space. The methods used in both examples are the same, however the exposition is more straightforward for the affine group. Therefore, the affine group is presented first and the Poincar\'{e} group second, following closely the template laid out by the affine group's example. In both cases the hardest part of the proof is proving the bijection result for the centralizer subgroups.

Our fourth section demonstrates a method for proving that two orbits in bijection with each other are homotopy equivalent. This relies on showing that the bijected orbits corresponding to the little subgroups and centralizer group orbits satisfy a property which we call being \emph{zigzag related}. This is an equivalence relation defined on homogeneous spaces which is stronger than that of being homotopy equivalent. Using this method we show that an adjoint and coadjoint orbit of the Poincar\'{e} group corresponding under the bijection are homotopy equivalent to each other.

We end with some remarks concerning the wider applicability of our methods.

\section{A bijection between orbits}
\subsection{The coadjoint orbits}

Let \(H\) be some Lie group, \(V\) a representation of this group, and \(G\) the semidirect product \(H\ltimes V\). The dual \(\mathfrak{g}^*\) is canonically isomorphic to \(\mathfrak{h}^*\times V^*\) and the coadjoint action given by \cite{rawnsley}
\begin{equation}
\label{coad_action}
\Ad_{(r,d)}^*(L,p)=\left(\Ad_r^*L+\mu(r^*p,d),r^*p\right).
\end{equation}
Here \((r,d)\in G\) and \((L,p)\in\mathfrak{h}^*\times V^*=\mathfrak{g}^*\). The map \(\mu\colon V^*\times V\longrightarrow\mathfrak{h}^*\) is defined by
\[
\langle \mu(p,v),\omega\rangle=\langle p,\omega v\rangle
\]
for all \(\omega\in\mathfrak{h}\). The subgroup \(H_p=\{r~|~r^*p=p\}\) is referred to in the literature as the \emph{little group}. The Lie algebra \(\mathfrak{h}_p\) of this group satisfies \(\mu(p,V)=\mathfrak{h}_p^\circ\), where \(\mathfrak{h}_p^\circ\) denotes the annihilator of the subalgebra.  We now reproduce the result given in \cite{rawnsley} which effectively parametrises the orbits in \(\mathfrak{g}^*\) by an orbit, say through \(p\) in \(V^*\), together with a \emph{little-group orbit} in \(\mathfrak{h}_p^*\).

Consider the set
\begin{equation}\label{Pi_defn}
\Pi=\left\{(l,p)~|~l\in\mathfrak{h}_p^*,~p\in V^*   \right\},
\end{equation}
and a coadjoint orbit \(\orb^*\) in \(\mathfrak{g}^*\). There is a map \(\mathfrak{g}^*\longrightarrow\Pi\) given by sending \((L,p)\) to \((\iota^*_pL,p)\), where \(\iota_p^*\) is the canonical projection of \(\mathfrak{h}^*\) onto \(\mathfrak{h}_p^*\). Let \(Y\) denote the image of \(\orb^*\) under this map. We may define an action of \(H\) on \(\Pi\) by setting \(r(l,p)=(rl,r^*p)\), where \(rl\) is the form in \(\mathfrak{h}_{rp}^*\) given by satisfying 
\[
\langle rl,\omega\rangle=\langle l,\Ad_{r^{-1}}\omega\rangle
\]
for all \(\omega\in\mathfrak{h}_{rp}\). Since \(\iota^*_p\colon\mathfrak{h}^*\longrightarrow\mathfrak{h}^*_p\) commutes with the action of \(H_p\) and has kernel \(\ker\iota^*_p=\mathfrak{h}_p^\circ\), the space \(Y\) is an orbit of \(H\) in \(\Pi\), and the map \(\orb^*\longrightarrow Y\) becomes an \(H\)-equivariant bundle with affine fibres \(\mathfrak{h}_p^\circ\). The space \(Y\) is also an \(H\)-equivariant bundle over an orbit in \(V^*\) given by projecting \((l,p)\) onto the second factor. The fibre of this projection above \(p\) is the coadjoint orbit through \(l\in\mathfrak{h}_p^*\), the so-called \emph{little-group orbit}. Conversely, given an orbit \(Y\) in \(\Pi\) there exists a coadjoint orbit \(\orb^*\subset\mathfrak{g}^*\) which is mapped to \(Y\). For \((l,p)\in Y\) the corresponding coadjoint orbit is that through \((L,p)\), where \(L\) is any element with \(\iota_p^*L=l\). In the literature the orbits of \(\Pi\) are referred to as \emph{bundles of little-group orbits}.
\begin{thm}[\cite{rawnsley}]\label{rawns}
	There is a bijection between the set of coadjoint orbits of \(G\) and the set of orbits of \(\Pi\). Given an orbit \(\orb^*\) and corresponding bundle \(Y\), there is a \(G\)-equivariant affine bundle \(\orb^*\rightarrow Y\).
\end{thm}

\subsection{The adjoint orbits}

We can adapt the bundle-of-little-group-orbits construction to the adjoint action, and obtain an analogous classification of orbits in \(\mathfrak{g}\). For \((\omega,v)\in\mathfrak{g}\) the adjoint action is \cite[Section 19]{stern}
\begin{equation}\label{Ad_action}
\Ad_{(r,d)}(\omega,v)=\left(\Ad_r\omega,rv-\left(\Ad_r\omega\right)d\right).
\end{equation}
The isotropy subgroup \(H_\omega=\{r~|~\Ad_r\omega=\omega\}\) is called the \emph{centralizer group}.
\begin{lem}\label{lemma}
	There is a canonical isomorphism between the quotient space \(V/\Imag\omega\) and \(\left(\ker\omega^*\right)^*\) (here \(\omega^*\) denotes the adjoint of the linear map \(\omega\)). Furthermore, this is an intertwining map for the representations of \(H_\omega\) on these spaces.
\end{lem}
\begin{proof}
	The result follows by dualizing the exact sequence, \(\ker\omega^*\hookrightarrow V^*\xrightarrow{\omega^*}V^*\), whose arrows all commute with \(H_\omega\).
\end{proof}
Consider the set
\begin{equation}\label{Sigma_defn}
\Sigma=\left\{(\omega,x)~|~\omega\in\mathfrak{h},~x\in\left(\ker\omega^*\right)^*\right\}.
\end{equation}
There is a map \(\mathfrak{g}\longrightarrow\Sigma\) which sends \((\omega,v)\) to \((\omega,x)\), where \(x\) is the element mapped from \([v]\in V/\Imag\omega\)  under the isomorphism in Lemma~\ref{lemma}. Now let \(\orb\) be an adjoint orbit and let \(X\) denote the image of this orbit under the map into \(\Sigma\). In the same way as we did for the coadjoint orbits, we can equip \(\Sigma\) with an \(H\)-action and establish \(\orb\longrightarrow X\) as an \(H\)-equivariant bundle with fibres \(\Imag\omega\). The space \(X\) is itself an \(H\)-equivariant bundle over an adjoint orbit through, say \(\omega\) in \(\mathfrak{h}\), with fibre above \(\omega\) equal to an orbit of \(H_\omega\) in \(\left(\ker\omega^*\right)^*\), what we shall call a \emph{centralizer group orbit}. In the same way that Theorem~\ref{rawns} is proven, we can establish an analogous theorem.
\begin{thm}
	There is a bijection between the set of adjoint orbits of \(G\) and the set of orbits of \(\Sigma\). Given an orbit \(\orb\) and corresponding bundle \(X\), there is a \(G\)-equivariant affine bundle \(\orb\rightarrow X\).
\end{thm}

\subsection{Constructing the bijection}
Suppose we have an action of a group \(G\) on \(X\), and of \(H\) on \(Y\). We will not give a precise meaning to the existence of a `geometric orbit bijection' between the two actions. For us an \emph{orbit bijection} will merely mean a bijection between the sets of orbits in \(X\) with orbits in \(Y\). This is a weak notion and does not capture the `geometric' sense of an orbit bijection as that given in the example from Figure~\ref{aff1}. To justify more rigorously what a `geometric' orbit bijection might mean would be a digressive and ultimately unnecessary exercise; the geometric nature of our bijections (whatever that may mean) will be clear from the construction we now give and from the examples to follow.

As we have seen, there is a bijection between the set of adjoint orbits of \(G\) with the set of \(H\)-orbits in \(\Sigma\), and a bijection between the coadjoint orbits of \(G\) with the \(H\)-orbits in \(\Pi\). Our strategy for showing an adjoint and coadjoint orbit bijection will be to exhibit a space \(\Delta\) equipped with an action of \(H\) for which there is an orbit bijection with both \(\Pi\) and \(\Sigma\).

Consider the diagonal action of \(H\) on the product \(\mathfrak{h}\times V^*\). We introduce the \(H\)-invariant subset \(\Delta\subset\mathfrak{h}\times V^*\) given by the three equivalent definitions
\begin{equation}\label{Delta_defn}
\Delta\colon=\left\{(\omega,p)~|~\omega^*p=0\right\}=\left\{(\omega,p)~|~\omega\in\mathfrak{h}_p\right\}=\left\{(\omega,p)~|~p\in\ker\omega^* \right\}.
\end{equation}
Observe that any orbit in \(\Delta\) is two different \(H\)-equivariant bundles given by projecting onto either the first or second factor. On the one hand, an orbit is a bundle over an orbit through \(p\in V^*\) with fibre equal to the adjoint orbit through \(\omega\in\mathfrak{h}_p\). On the other hand, it is also a bundle over the adjoint orbit through \(\omega\in\mathfrak{h}\) with fibre equal to the \(H_\omega\)-orbit through \(p\in\ker\omega^*\).
\begin{thm}[Orbit bijection]\label{orb_bijection}
	Suppose that for any \(p\in V^*\) there exists a bijection between the set of adjoint and coadjoint orbits of \(H_p\). Additionally, suppose there is a bijection between the set of \(H_\omega\)-orbits on \(\ker\omega^*\) with the set of \(H_\omega\)-orbits of the contragredient representation on \((\ker\omega^*)^*\). Then there exists a bijection between the set of orbits in \(\Delta\), with each of the sets of adjoint and coadjoint orbits of \(G\).
\end{thm}
\begin{proof}
	Any coadjoint orbit of \(G\) uniquely determines an orbit \(Y\) through, say \((l,p)\) in \(\Pi\). The space \(Y\) is a bundle over the orbit through \(p\in V^*\) whose fibre over \(p\) is the coadjoint orbit through \(l\in\mathfrak{h}_p^*\). Contrast this with an orbit through \((\omega,p)\) in \(\Delta\): a bundle over the orbit through \(p\in V^*\) whose fibre over \(p\) is an adjoint orbit in \(\mathfrak{h}_p\). Since there is a bijection between adjoint and coadjoint orbits of \(H_p\), let \(\omega\) be any element belonging to the adjoint orbit which is in bijection with the coadjoint orbit through \(l\). We designate the orbit \(Z\) through \((\omega,p)\) to correspond to the coadjoint orbit we selected at the beginning.
	
	This correspondence currently depends on which point \(p\) we select in the orbit through \(V^*\). For instance, had we taken the point \((rl,rp)\in Y\) instead, then the bijection between adjoint and coadjoint orbits of \(H_{rp}\) may not result in the same choice of designated orbit in \(\Delta\) as it did for \(H_p\). To ward against this we make an additional assumption about our bijections. For a given \(p\in V^*\) let \(\orb\) and \(\orb^*\) be orbits in \(\mathfrak{h}_p\) and \(\mathfrak{h}_p^*\) respectively which are in bijection with each other. We insist that for any \(r\in H\) the bijection between adjoint and coadjoint orbits of \(H_{rp}\) is given by bijecting the adjoint orbit \(\Ad_r\orb\) with the coadjoint orbit \(\Ad_{r^{-1}}^*\orb^*\) (note that \(\Ad_r\colon\mathfrak{h}_p\longrightarrow\mathfrak{h}_{rp}\) is an isomorphism).
	
	With this assumption on the `consistency' of the \(H_p\)-orbit bijections, the correspondence we described above no longer depends on the point \((l,p)\in Y\), but instead only depends on the orbit \(Y\) itself. In this way we define a bijection between the orbits of \(\Pi\) with the orbits of \(\Delta\).
	
	The proof for the adjoint orbits is analogous so we will merely sketch it. It suffices to define a bijection between orbits in \(\Sigma\) with those in \(\Delta\). Orbits in both these spaces are bundles over adjoint orbits in \(\mathfrak{h}\) but whose fibres are either \(H_\omega\)-orbits in \(\ker\omega^*\) or its dual. Making similar assumptions about the \(H_\omega\)-bijections being `consistent' we may replicate the construction given for the coadjoint orbits to define an orbit-to-orbit bijection between \(\Sigma\) and \(\Delta\).
\end{proof}

\section{The affine group}
\subsection{Preliminaries}
The \emph{affine group} \(\aff(V)\) of a vector space \(V\) is the group of affine linear transformations of \(V\). It is a semidirect product
\[
\aff(V)=GL(V)\ltimes V
\]
with respect to the defining representation of \(GL(V)\). If we choose a basis for \(V\) then this group is isomorphic to the matrix group
\begin{equation}\label{rows}
\left\{
\begin{pmatrix}
r & d \\ 0 & 1
\end{pmatrix}~\Big|~ r\in GL(V),~d\in V
\right\}.
\end{equation}
If we identify \(V\) with its dual using the standard inner product corresponding to our choice of basis, then the group \(\aff(V^*)\) is isomorphic to
\begin{equation}\label{column}
\left\{
\begin{pmatrix}
r^{-1} & 0 \\ d^T & 1
\end{pmatrix}~\Big|~ r\in GL(V),~d\in V
\right\}.
\end{equation}
For when \(V=\mathbb{R}^n\) we write \(\aff(V)=\aff(n)\) and \(\aff(V^*)=\aff(n^*)\). Observe that the Lie algebra \(\mathfrak{aff}(n)\) is isomorphic to the set of \((n+1)\) square matrices whose final row consists of zeros, and \(\mathfrak{aff}(n^*)\) to the set of \((n+1)\) square matrices whose final column consists of zeros.
\begin{propn}\label{aff_stab}
	For \(v\in \mathbb{R}^{n+1}\) non-zero we have that the isotropy subgroup \(GL(n+1)_v\) is isomorphic to \(\aff(n^*)\). Similarly, for \(p\) a non-zero linear functional on \(\mathbb{R}^{n+1}\), the subgroup \(GL(n+1)_p\) is isomorphic to \(\aff(n)\).
\end{propn}
\begin{proof}
	Since \(GL(n+1)\) acts transitively on non-zero vectors we may suppose without loss of generality that \(v\) is equal to the final basis vector for our choice of basis. The subgroup which preserves this vector is then precisely that given by equation~\eqref{column}. The dual action on \(\mathbb{R}^{n+1}\) is given by right multiplication on row vectors. The stabiliser of the final basis element now corresponds to matrices as in equation~\eqref{rows}.
\end{proof}
\begin{ex}[Orbits of $\aff(1)$]\label{aff1_ex}
	Denote group elements by \((r,d)\in GL(1)\ltimes\mathbb{R}^1=\aff(1)\) and Lie algebra elements by \((\omega,v)\in\mathfrak{gl}(1)\times\mathbb{R}^1=\mathbb{R}\times\mathbb{R}\). We will identify \(\mathfrak{aff}(1)\) with its dual by setting \(\langle (L,p),(\omega,v)\rangle=L\omega+pv\). It can be shown that the adjoint action (see \eqref{Ad_action}) is given by 
	\[
	\Ad_{(r,d)}(\omega,v)=\left(\omega,rv-\omega d\right)
	\]
	and the coadjoint action (see \eqref{coad_action}) by
	\[
	\Ad_{(r,d)}^*(L,p)=\left(L+r^{-1}pd,r^{-1}p\right).
	\]
	In Figure~\ref{aff1} we illustrate these orbits. Observe that there is indeed a `geometric' bijection between the orbits: both origins to each other, the full-line adjoint orbits to the remaining coadjoint point orbits, and the two-half-line adjoint orbit to the open, dense coadjoint orbit.
\end{ex}
\begin{figure}
	\centering
	\begin{subfigure}{0.4\textwidth}
		\caption{Adjoint orbits}
		\begin{tikzpicture}
		\path[use as bounding box] (-3.3,-3.3) rectangle (3.5,3.3);
		\draw[-Latex] (0,-3) -- (0,3)node[right]{$v$};
		\draw[-Latex] (-3,0) -- (3,0)node[right]{$\omega$};
		\draw[fill] (0,0) circle[radius=0.07];
		\draw[ultra thick] (0,0.2) -- (0,2.5);
		\draw[ultra thick] (0,-0.2) -- (0,-2.5);
		
		\foreach \x in {0.5,1,...,2.5}{
			\draw[ultra thick] (\x,-2.5) -- (\x,2.5);
			\draw[ultra thick] (-\x,-2.5) -- (-\x,2.5);	
		}
		\end{tikzpicture}
	\end{subfigure}
	\begin{subfigure}{0.4\textwidth}
		\caption{Coadjoint orbits}
		\begin{tikzpicture}
		\path[use as bounding box] (-3.3,-3.3) rectangle (3.5,3.3);
		\draw[-Latex] (0,-3) -- (0,3)node[right]{$p$};
		\draw[-Latex] (-3,0) -- (3,0)node[right]{$L$};
		\draw[fill] (0,0) circle[radius=0.07];
		\draw[draw=gray,fill=gray, fill opacity=0.4,draw opacity=0.4] (-2.5,0.2) rectangle (2.5,2.5);
		\draw[draw=gray,fill=gray, fill opacity=0.4,draw opacity=0.4] (-2.5,-0.2) rectangle (2.5,-2.5);
		\foreach \x in {0.5,1,...,2.5}{
			\draw[fill] (\x,0) circle[radius=0.07];
			\draw[fill] (-\x,0) circle[radius=0.07];
		}
		\end{tikzpicture}
	\end{subfigure}
	\caption{\label{aff1}Orbits of \(\aff(1)\).}
\end{figure}
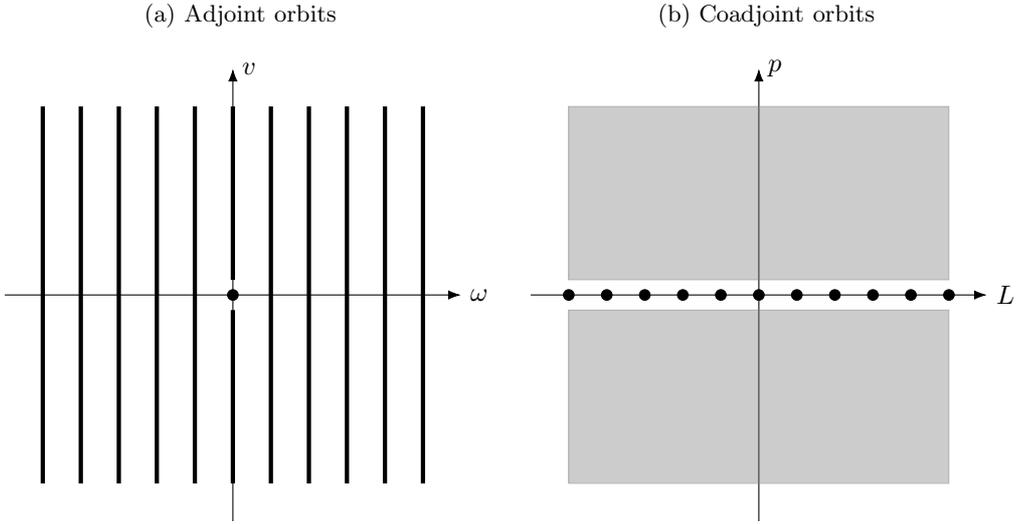

\subsection{The centralizer group representation}
For Theorem~\ref{orb_bijection} to hold we require a bijection result for the \(H_p\)- and \(H_\omega\)-orbits. We will begin with the representation of \(H_\omega\) on \(\ker\omega^*\) and its dual. 

Let \(\Phi\) denote the representation 
\[
\Phi\colon H_\omega\longrightarrow GL(\ker\omega^*)
\]
given by restriction of \(r\in H_\omega\) to \(\ker\omega^*\). The group \(H_\omega=GL(n)_\omega\) is the subgroup of all isomorphisms of \(V\) which commute with \(\omega\). Equivalently, it is the subgroup of all isomorphisms of \(V^*\) which commute with \({\omega^*}\). Observe that \(H_\omega\) must therefore preserve the flag
\begin{equation}
\label{nilp_flag}
F_{{\omega^*}}=\left(\ker{\omega^*}\supset \Imag{\omega^*}\cap\ker{\omega^*}\supset\Imag{\omega^*}^2\cap\ker{\omega^*}\supset\dots\right).
\end{equation} 

\begin{propn}\label{flag_propn_aff}
	The group \(\Phi(H_\omega)\) is the group of all isomorphisms of \(\ker{\omega^*}\) which preserve the flag \(F_{\omega^*}\).
\end{propn}
\begin{proof}
Define the \emph{length} \(l=l(F_{\omega^*})\) of the flag \(F_{\omega^*}\) to be the least positive integer with \(\Imag{\omega^*}^l\cap\ker\omega^*=\{0\}\) . The length of \(F_{\omega^*}\) is finite and the proposition is trivially true for when \(l=0\). For when \(l=1\) we have the direct sum decomposition \(V^*=\Imag{\omega^*}\oplus\ker{\omega^*}\) as  \(\Imag{\omega^*}\cap\ker{\omega^*}=\{0\}\). The proposition is true for this case as we may set \(r\) to be the identity on \(\Imag\omega^*\) and equal to any isomorphism on \(\ker\omega^*\). Suppose for induction that the result is true for all \(\overline{\omega}^*\) with \(l(F_{\overline{\omega}^*})<l(F_{\omega^*})\). For a given isomorphism \(r\) of \(\ker\omega^*\) preserving \(F_{\omega^*}\) it suffices to show that it may be extended to an isomorphism of \(V^*\) which commutes with \(\omega^*\).

Let \(\overline{\omega}^*\) denote the restriction of \({\omega^*}\) to \(\Imag{\omega^*}\). The flag \(F_{\overline{\omega}^*}\) is equal to the subflag
\[
\Imag{\omega^*}\cap\ker{\omega^*}\supset\Imag {\omega^*}^2\cap\ker{\omega^*}\supset\dots
\]
of \(F_{\omega^*}\), and therefore \(l(F_{\overline{\omega}^*})<l(F_{\omega^*})\). By the induction hypothesis, we may extend the definition of \(r\) to an isomorphism of \(\Imag{\omega^*}\) in such a way that it commutes with \({\omega^*}\).

Now that we have defined \(r\) on \(\Imag{\omega^*}+\ker{\omega^*}\), let \(y_1,\dots,y_m\) be vectors in \(V^*\) such that \(\{[y_1],\dots,[y_m]\}\) is a basis of \(V^*/\Imag{\omega^*}+\ker{\omega^*}\). We define each \(ry_i\) to be equal to any element of \(V^*\) which satisfies \({\omega^*} ry_i=r{\omega^*} y_i\), and extend the definition linearly over the \(y_i\)s. Observe that \(r\) commutes with the isomorphism
\begin{equation}
V^*/(\Imag{\omega^*}+\ker{\omega^*})\overset{\cong}{\longrightarrow}\Imag{\omega^*}/\Imag{\omega^*}^2
\end{equation}
given by sending the class \([v]\) to \([{\omega^*} v]\). We therefore have extended \(r\) to an isomorphism over all of \(V^*\) which commutes with \({\omega^*}\) as desired.
\end{proof}
\subsection{Establishing the orbit bijection}
We are now in a position to show that there is a bijection between the set of orbits of \(H_\omega\) on \(\ker\omega^*\) with the set of \(H_\omega\)-orbits on \((\ker\omega^*)^*\). We begin by rewriting the flag \(F_{\omega^*}\) in \eqref{nilp_flag} as a strictly descending sequence of subspaces starting at \(\ker\omega^*\subset V^*\) and terminating at \(\{0\}\)
\begin{equation*}
F_{\omega^*}=\left(\ker\omega^*=E_0\supsetneq E_1\supsetneq E_2\supsetneq\dots\supsetneq E_{k+1}=\{0\}\right).
\end{equation*}
The previous proposition tells us that \(H_\omega\) acts on \(\ker\omega^*\) by all isomorphisms preserving this flag. Therefore, there are precisely \((k+2)\) orbits given by the sets
\begin{equation}\label{flag_orbits}
E_{0}\setminus E_1,~E_1\setminus E_2,~\dots,~E_k\setminus E_{k+1},~E_{k+1}=\{0\}.
\end{equation}
The contragredient representation of \(H_\omega\) on \((\ker\omega^*)^*\) must act by all isomorphisms which preserve the \emph{dual flag} \(F_{{\omega^*}}^\circ\) given by the ascending sequence of annihilators
\begin{equation}\label{dual_nilp_flag}
F^\circ_{\omega^*}\coloneqq\left(\{0\}\subsetneq E_1^\circ\subsetneq E_{2}^\circ,\dots,E_k^\circ\subsetneq E_{k+1}^\circ=(\ker\omega^*)^*\right).
\end{equation}
There are thus \((k+2)\) orbits in \((\ker\omega^*)^*\) given by the sets
\begin{equation}\label{dual_flag_orbits}
E_0^\circ=\{0\},~E_{1}^\circ\setminus E_0^\circ,~,~E^\circ_2\setminus E^\circ_1,~\dots,~ E_{k+1}^\circ\setminus E_k^\circ.
\end{equation}
We define the bijection between the sets of orbits in \eqref{flag_orbits} and \eqref{dual_flag_orbits} to be given by
\begin{equation}\label{flag_bij}
E_{k+1}\longleftrightarrow E_0^\circ,~\text{and  }(E_j\setminus E_{j+1})\longleftrightarrow (E_{j+1}^\circ\setminus E_j^\circ)~\text{for  }0\le j\le k.
\end{equation}
\begin{thm}\label{aff_bij}
	[Affine-group orbit bijection] There is a bijection between the set of \(GL(n)\)-orbits through \(\Delta\) with each of the sets of adjoint and coadjoint orbits of \(\aff(n)\).
\end{thm}
\begin{proof}
	The groups \(H_p\) are isomorphic to either \(GL(n)\) if \(p=0\) or \(\aff(n-1)\) for \(p\ne 0\) (see Proposition~\ref{aff_stab}). For \(GL(n)\) the adjoint and coadjoint representations are isomorphic, therefore there is trivially a bijection between the two sets of orbits. For \(\aff(n-1)\) we suppose for induction that the result is true, noting from Figure~\ref{aff1} that this is true for \(\aff(1)\). The \(H_\omega\)-orbit bijection is given in \eqref{flag_bij} and thus the theorem follows by a direct application of Theorem~\ref{orb_bijection} together with induction on \(n\).
\end{proof}

\subsection{An iterative method for obtaining orbit types}
Consider the set of orbits in \(\Delta=\Delta_n\) for \(G=\aff(n)\). The set of orbits through \((\omega,0)\in\Delta\) is in bijection with the set of adjoint orbits of \(\mathfrak{h}=\mathfrak{gl}(n)\). Since the action of \(H\) on non-zero vectors in \(V^*\) is transitive, the remaining orbits are those through the points \((\omega,p)\), for some fixed non-zero \(p\). The set of all such orbits is in bijection with the set of adjoint orbits of \(\mathfrak{h}_p\), which by Proposition~\ref{aff_stab} is isomorphic to \(\affo(n-1)\). By Theorem~\ref{aff_bij}, the set of adjoint orbits in \(\affo(n-1)\) is itself in bijection with the set of \(GL(n-1)\)-orbits through \(\Delta_{n-1}\). We may apply the same argument iteratively to obtain a hierarchy of orbit types as demonstrated in Figure~\ref{aff_tree}

\begin{thm}
	The set of orbits in \(\Delta\) for \(G=\aff(n)\) with \(n>1\), is in bijection with the set of adjoint orbits of \(GL(k)\) for \(2\le k\le n\) and \(\aff(1)\).
\end{thm}

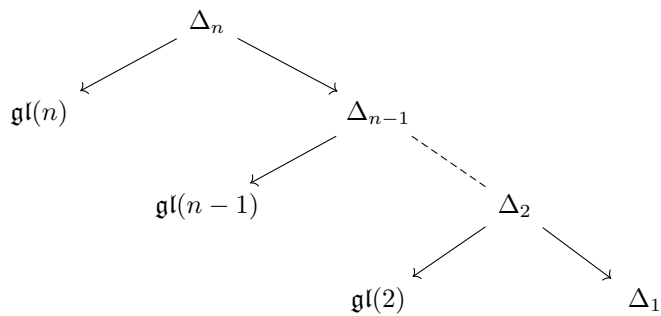
\begin{figure}
	\centering
	\begin{tikzcd}
		& \Delta_n \arrow[ld]\arrow[rd] &  &  &  \\
		\mathfrak{gl}(n) &  & \Delta_{n-1} \arrow[ld]\arrow[rd,dash,dashed] &  &  \\
		& \mathfrak{gl}(n-1) &  & \Delta_2 \arrow[ld]\arrow[rd] &  \\
		&  & \mathfrak{gl}(2) &  & \ \Delta_1
	\end{tikzcd}
	\caption{\label{aff_tree}Hierarchy of orbit types for \(\aff(n)\).}
\end{figure}
\section{The Poincar\'{e} group}
\subsection{Preliminaries}
Let \(V\) be a real \(n\)-dimensional vector space equipped with a non-degenerate, symmetric bilinear form \(Q\) of signature \((m,n)\). Vectors \(v\) are partitioned into three sets depending on the value of \(Q(v,v)\): if it is positive \(v\) is said to be \emph{timelike}, \emph{spacelike} if negative, and \emph{null} if it is zero. All such vectors spaces are isomorphic to \emph{Minkowski space} \(\mathbb{R}^{m+n}\). Letting \(\{e_1,\dots,e_m,f_1,\dots,f_n\}\) denote the standard basis of \(\mathbb{R}^{m+n}\) the bilinear form is given by \(Q(e_i,f_j)=0\), \(Q(e_i,e_j)=+\delta_{ij}\), and \(Q(f_i,f_j)=-\delta_{ij}\).

The \emph{indefinite orthogonal group} or \emph{Lorentz group} \(O(V;Q)\) is the group of isomorphisms of \(V\) which preserve \(Q\). The \emph{Poincar\'{e}} group is the semidirect product \(E(V;Q)=O(V;Q)\ltimes V\). For Minkowski space we write the Lorentz and Poincar\'{e} groups as \(O(m,n)\) and \(E(m,n)\), and their Lie algebras by \(\mathfrak{so}(m,n)\) and \(\mathfrak{se}(m,n)\) respectively.

\begin{propn}\label{sopq_stabilisers}
	Let \(\tau,\sigma\) and \(\nu\) be non-zero vectors in \(\mathbb{R}^{m+n}\) which are timelike, spacelike and null respectively. Then we have isomorphisms: \(O(m,n)_\tau\cong O(m-1,n)\), \(O(m,n)_\sigma\cong O(m,n-1)\) and \(O(m,n)_\nu\cong E(m-1,n-1)\).
\end{propn}
\begin{proof}
	The case for the timelike and spacelike vectors follows from the fact that the orthogonal complement to these vectors has signature \((m-1,n)\) and \((m,n-1)\) respectively and is an invariant subspace. For the non-zero null vector, this is proven in \cite[Section 2]{cush06}.
\end{proof}
\begin{ex}[Orbits of $E(1,1)$]\label{poinc11}
	The group \(O(1,1)\) is the group of \emph{Lorentz boosts}, which we write as \(r_\psi\) (where \(\psi\) denotes the \emph{rapidity}). Elements in \(E(1,1)\) will be denoted by \((r_\psi,d)\in O(1,1)\ltimes\mathbb{R}^{1+1}\) and in the Lie algebra by \((\omega,v)\in\mathfrak{so}(1,1)\times\mathbb{R}^{1+1}\). The adjoint action can be shown to be (see \eqref{Ad_action})
	\[
	\Ad_{(r_\psi,d)}(\omega,v)=\left(\omega,r_\psi v-\omega d\right).
	\]
	Identify \(\mathfrak{se}(1,1)\) with its dual by setting \(\langle (L,p),(\omega,v)\rangle=\Tr(L^T\omega)+p^TQv\) where now \(Q\) denotes the matrix \(\text{diag}(1,-1)\). The coadjoint action may be shown to be (see \eqref{coad_action})
	\[
	\Ad_{(r_\psi,d)}^*(L,p)=\left(L+\mu(r_\psi p,d),r_\psi p\right),
	\]
	where \(\mu(p,v)=(pv^T-vp^T)Q\). Here there exists the following orbit bijection: both origins to each other; adjoint orbits through \((\omega,v)\) for \(\omega\ne 0\) to coadjoint orbits through \((L,p)=(\omega,0)\); and coadjoint orbits through \((L,p)\) for \(p\ne 0\) to adjoint orbits through \((\omega,v)=(0,p)\). 
\end{ex}
\subsection{The centralizer group representation}

As with the affine group, the hardest part of our method will be finding an orbit bijection for the centralizer orbits of \(H_\omega\) on \(\ker\omega^*\) and its dual. From now on, a \emph{form} on a vector space will mean a symmetric or skew-symmetric bilinear form. 

Let \(H\) be the group \(O(V;Q)\) of all isomorphisms of \(V\) which preserve a given non-degenerate form \(Q\). For a given element \(\omega\) in the Lie algebra \(\mathfrak{h}\) of \(H\) we consider the centralizer subgroup \(H_\omega=\{r\in H~|~r\omega=\omega r\}\) and the representation 
\[
\Phi\colon H_\omega\longrightarrow GL(\ker\omega^*)
\]
given by restricting \(r\) to \(\ker\omega^*\). As \(\omega\) is a skew-self-adjoint operator with respect to \(Q\) we may identify \(V\) with \(V^*\) and \(\omega\) with \(-\omega^*\) and from now on consider the action of \(H_\omega\) on \(\ker\omega\). The group \(H_\omega\) is the subgroup of all isomorphisms preserving \(Q\) which commute with \(\omega\) and therefore the following flag is preserved.
\begin{equation}
\label{ortho_flag}
F_{\omega}=\left(\ker\omega\supset \Imag\omega\cap\ker\omega\supset\Imag\omega^2\cap\ker\omega\supset\dots\right).
\end{equation}
\begin{propn}\label{quotient_propn}
	 In addition to preserving the flag \(F_\omega\) the group \(\Phi(H_\omega)\) also preserves non-degenerate forms defined on the quotient spaces 
	\begin{equation}
	\label{flag_quotients}
	\Imag\omega^{m}\cap\ker\omega/\Imag\omega^{m+1}\cap\ker\omega
	\end{equation}
	for \(m\ge 0\) and where \(\omega^0\) is the identity. For \(\omega^mx\) and \(\omega^my\) belonging to \(\Imag\omega^m\cap\ker\omega\), the form is given on the quotient space in equation~\eqref{flag_quotients} by
	\begin{equation}
	\label{quotient_forms}
	\langle[\omega^mx],[\omega^my]\rangle\coloneqq Q(\omega^mx,y).
	\end{equation}
\end{propn}
\begin{proof}
	Suppose for induction that \(\overline{V}=\Imag\omega^{k}\) is equipped with a non-degenerate form given by \(\overline{Q}(\omega^{k}x,\omega^{k}y)=Q(\omega^{k}x,y)\); this is true for \(k=0\). Let \(\overline{\omega}\) denote the restriction of \(\omega\) to \(\overline{V}\).
	
	Observe that \(\overline{\omega}\) is skew-self-adjoint with respect to \(\overline{Q}\) and so \((\ker\overline{\omega})^\perp=\Imag\overline{\omega}\). Therefore \(\Imag\overline{\omega}\cap\ker\overline{\omega}\) is the null space for \(\ker\overline{\omega}\). It follows that the quotient \(\ker\overline{\omega}/\Imag\overline{\omega}\cap\ker\overline{\omega}\) inherits a non-degenerate form by restriction of \(\overline{Q}\). By noting that \(\Imag\overline{\omega}=\Imag\omega^{k+1}\) and \(\ker\overline{\omega}=\Imag\omega^{k}\cap\ker\omega\), we see that we have proved the result for \(m=k+1\).
	
	The image \(\Imag\overline{\omega}\) is also equipped with a non-degenerate form given by \(\overline{\overline{Q}}(\overline{\omega} a,\overline{\omega} b)=\overline{Q}(\overline{\omega}a,b)\). Recalling that \(\Imag\overline{\omega}=\Imag\omega^{k+1}\) and writing \(a=\omega^k x\) and \(b=\omega^k y\), we therefore see that \(\Imag\omega^{k+1}\) is equipped with a non-degenerate form given by \(\overline{\overline{Q}}(\omega^{k+1}x,\omega^{k+1}y)=Q(\omega^{k+1}x,y)\), and thus the result is proven by induction on \(k\).
\end{proof}

\begin{thm}\label{the_worst}
	The subgroup \(\Phi(H_\omega)\subset GL(\ker\omega)\) is precisely equal to the group of all isomorphisms of \(\ker\omega\) which preserve the flag \(F_\omega\) together with all the forms given in equation~\eqref{quotient_forms} on the quotient spaces.
\end{thm}
\begin{proof}
	For when \(l(F_\omega)=1\) we have \(\Imag\omega\cap\ker\omega=\{0\}\) and thus we have the orthogonal, direct sum decomposition \(V=\Imag\omega\oplus\ker\omega\). The theorem is true in this case as we may define \(r\) to be the identity on \(\Imag\omega\), and any isomorphism on \(\ker\omega\) which preserves \(Q\). Suppose for induction that the theorem is true for all triples \((\overline{V},\overline{Q},\overline{\omega})\) with \(l(F_{\overline{\omega}})<l(F_\omega)\). We begin by fixing an isomorphism \(r\) of \(\ker\omega\) which preserves \(F_\omega\) together with all of its forms. It suffices to show that \(r\) may be extended to an isomorphism of \(V\) which preserves \(Q\) and commutes with \(\omega\).
	
	Let \(\overline{\omega}\) denote the restriction of \(\omega\) to \(\overline{V}=\Imag\omega\) and recall that this is equipped with the non-degenerate form \(\overline{Q}(\omega x,\omega y)=Q(\omega x,y)\). We leave it as an exercise to show that the flag \(F_{\overline{\omega}}\) is equal to the subflag
	\[
	\Imag\omega\cap\ker\omega\supset\Imag\omega^2\cap\ker\omega\supset\dots
	\]
	of \(F_\omega\) and that moreover, the forms on the quotient spaces coincide. By the induction hypothesis, since \(l(F_{\overline{\omega}})<l(F_\omega)\), we can extend \(r\) to an isomorphism of \(\Imag\omega\) which commutes with \(\omega\) and preserves \(\overline{Q}\).
	
	Now that we have defined \(r\) on \(\Imag\omega+\ker\omega\), let \(y_1,\dots,y_m\) be vectors in \(V\) such that \(\{[y_1],\dots,[y_m]\}\) is a basis of \(V/\Imag\omega+\ker\omega\). We define each \(ry_i\) to be equal to any element of \(V\) which satisfies \(\omega ry_i=r\omega y_i\), and extend the definition linearly over the \(y_i\)s. Observe that \(r\) then commutes with the isomorphism
	\begin{equation}
	\label{usefule_isom}
	V/(\Imag\omega+\ker\omega)\overset{\cong}{\longrightarrow}\Imag\omega/\Imag\omega^2
	\end{equation}
	given by sending the class \([v]\) to \([\omega v]\). We therefore have extended \(r\) to an isomorphism over all of \(V\) which additionally preserves the form \(\overline{Q}\) on \(\Imag\omega\), and commutes with \(\omega\). This implies that \(Q(r\omega x,ry)=Q(\omega x,y)\) for all \(\omega x\in\Imag\omega\) and \(y\in V\). 
	
	It remains then to show that \(Q(rx,ry)=Q(x,y)\) for \(x\) and \(y\) not in \(\Imag\omega\). We will apply a Gram-Schmidt style procedure to alter the definition of each \(ry_i\) to force this to hold. Let \(x_1,\dots,x_m\) be vectors in \(\ker\omega\) such that \(\{[x_1],\dots,[x_l]\}\) is a basis of \(\ker\omega/\Imag\omega\cap\ker\omega\). Observe that \(\{[x_1],\dots,[x_l],[y_1],\dots,[y_m]\}\) is then a basis of \(V/\Imag\omega\). It therefore suffices to show that \(r\) preserves \(Q\) when restricted to the elements \(x_1,\dots,x_l,y_1,\dots,y_m\). 
	
	We begin by claiming that \(Q(rx_i,rx_j)=Q(x_i,x_j)\); this follows from the fact that \(r\) preserves the form \(Q\) restricted to the quotient \(\ker\omega/\Imag\omega\cap\ker\omega\). As the pairing between \(\ker\omega\) and \(V/\Imag\omega\) is non-degenerate, it follows that there exist \(k_i\in\ker\omega\) for each \(1\le i\le m\) which satisfy
	\[
	Q(k_i,rx_j)=Q(y_i,x_j)-Q(ry_i,rx_j)
	\]
	for all \(1\le j\le l\). Redefine each \(ry_i\) to now equal \(ry_i+k_i\) and extend linearly over the \(y_i\)s. Verify that \(Q(ry_i,rx_j)=Q(y_i,x_j)\) and \(\omega ry_i=r\omega y_i\) for all pairs of \(i\) and \(j\).

As the isomorphism in equation~\eqref{usefule_isom} commutes with \(r\) it follows that the elements \([\omega ry_1],\dots,[\omega ry_m]\) form a basis of \(\Imag\omega/\Imag\omega^2=\overline{V}/\Imag\overline{\omega}\). As the pairing between \(\ker\overline{\omega}=\Imag\omega\cap\ker\omega\) and \(\overline{V}/\Imag\overline{\omega}\) is non-degenerate, we may sequentially construct \(\omega z_i\in\ker\overline{\omega}\), for each \(1\le i\le m\), which satisfy
	\[
	\overline{Q}(\omega z_i,\omega ry_j)=
	Q(y_i,y_j)-Q(ry_i,ry_j)-Q(ry_i,\omega z_j)
	\]
	for all \(j<i\); and, if \(Q\) is symmetric, \(2\overline{Q}(\omega z_i,\omega ry_i)=Q(y_i,y_i)-Q(ry_i,ry_i)\). We once again alter the definition of \(ry_i\) and change it to equal \(ry_i+\omega z_i\). One can now verify that \(Q(ry_i,rx_j)=Q(y_i,x_j)\), \(Q(ry_i,ry_j)=Q(y_i,y_j)\), \(Q(ry_i,r\omega x)=Q(y_i,\omega x)\), and that \(r\omega y_i=\omega ry_i\) for all pairs of \(i\) and \(j\), and all \(\omega x\in\Imag\omega\). By extending the definition of \(r\) linearly over the \(y_i\)s we thus obtain an isomorphism of \(V\) which commutes with \(\omega\) and preserves \(Q\) as desired.
	\end{proof}

\subsection{Establishing the orbit bijection}
To apply Theorem~\ref{orb_bijection} and prove an orbit bijection result for the Poincar\'{e} group we need to demonstrate an orbit bijection for the little subgroups \(H_p\) and the centralizer groups \(H_\omega\). We begin by considering the representation of \(H_\omega\) on \(\ker\omega\) for where \(H=O(m,n)\). Rewrite the flag in equation~\eqref{ortho_flag} as a strictly descending sequence of subspaces
\begin{equation}\label{distinct_flag}
F_\omega=\left(\ker\omega=E_0\supsetneq E_1\supsetneq E_2\supsetneq\dots\supsetneq E_{k+1}=\{0\}\right).
\end{equation}
Let \(Q_j\) denote the non-degenerate form defined on each quotient \(E_j/E_{j+1}\) given by equation~\eqref{quotient_forms} in Proposition~\ref{quotient_propn}. From Theorem~\ref{the_worst}, the representation of \(H_\omega\) acts on \(\ker\omega\) by all isomorphisms which preserve \(F_\omega\) together with all of the forms \(Q_j\). We therefore see that any orbit which is not \(E_k=\{0\}\) is contained to a set of the form \(E_j\setminus E_{j+1}\) for \(0\le j\le k\). Suppose the element \(p\) belongs to such an orbit. We remark that any other element \(\widetilde{p}\) belonging to the same non-zero equivalence class of \([p]\in E_j/E_{j+1}\) also belongs to the same orbit. As \(H_\omega\) must preserve the form \(Q_j\) it follows that the orbit through \(p\) uniquely defines an orbit through \([p]\) of \(O(E_j/E_{j+1};Q_j)\). In summary then, the \(H_\omega\)-orbits correspond to the origin \(\{0\}\), and an integer \(0\le j\le k\) together with an orbit of \(O(E_j/E_{j+1};Q_j)\) in \(E_j/E_{j+1}\).

The contragredient representation of \(H_\omega\) on \((\ker\omega)^*\) must act by the group of all isomorphisms which preserve the dual flag
\[
F^\circ_\omega=\left(\{0\}\subsetneq E_1^\circ\subsetneq E_2^\circ\subsetneq\dots\subsetneq E_{k+1}^\circ=(\ker\omega)^*\right)
\]
together with the non-degenerate co-forms \(Q_j^*\) defined on each quotient \(E_{j+1}^\circ/E_j^\circ\cong\left(E_j/E_{j+1}\right)^*\). Repeating the argument in the previous paragraph, we see that the orbits are either the origin, or contained to the sets of the form \(E_{j+1}^\circ\setminus E_j^\circ\) for each \(0\le j\le k\), and that such an orbit determines an orbit of \(O(E_{j+1}^\circ/E_j^\circ;Q_j^*)\) in \(E_{j+1}^\circ/E_j^\circ\).

For any \(j\) with \(0\le j\le k\) consider the map \(\varphi\colon E_{j}\rightarrow E_{j+1}^\circ\) given by satisfying
\begin{equation}
\label{varphi_map}
\langle \varphi(p),q\rangle=Q_{j}([p],[q])
\end{equation}
for all \(p,q\in E_{j}\). The map \(\varphi\) descends to the quotient spaces to give an isomorphism \(\varphi\colon E_{j}/E_{j+1}\rightarrow E^\circ_{j+1}/E^\circ_{j}\) which preserves the forms; that is \(\varphi^*Q_{j}^*=Q_j\). This map therefore establishes the bijection.
\begin{propn}\label{ortho_flag_bij}
	There is a bijection between the set of \(H_\omega\)-orbits in \(\ker\omega\) with the set of \(H_\omega\)-orbits in \((\ker\omega)^*\). Any orbit through a non-zero \(p\in\ker\omega\) is contained to a set \(E_j\setminus E_{j+1}\) with respect to the flag \(F_\omega\) given in \eqref{distinct_flag} and for \(0\le j\le k\). The corresponding bijected orbit is that through \(\varphi(p)\in E_{j+1}^\circ\setminus E_j^\circ\) where \(\varphi\) is given in \eqref{varphi_map}. The orbits equal to the origin are both bijected with each other.
\end{propn}

\begin{thm}\label{poinc_bij}
	[Poincar\'{e}-group orbit bijection]  There is a bijection between the set of \(O(m,n)\)-orbits through \(\Delta\) with each of the sets of adjoint and coadjoint orbits of \(E(m,n)\).
\end{thm}
\begin{proof}
	To begin with, the \(H_\omega\)-orbit bijection follows from Proposition~\ref{ortho_flag_bij}. Therefore, in order to apply Theorem~\ref{orb_bijection} it remains to show that there is a bijection between the adjoint and coadjoint orbits of the \(H_p\) little groups.
	
	Since the vector representation of \(O(m,n)\) is isomorphic to its dual, the subgroups \(H_p\) for \(p\in V^*\) are isomorphic to the stabiliser subgroups for different vectors in \(V\). From Proposition~\ref{sopq_stabilisers} these groups are isomorphic to \(O(m,n)\), \(O(m-1,n)\), \(O(m,n-1)\) and \(E(m-1,n-1)\) for when \(p\) is zero, timelike, spacelike, and non-zero and null respectively (and for whenever the entries are non-negative). For the first three cases, these groups are semisimple, and thus the adjoint and coadjoint representations are isomorphic; consequently they trivially exhibit an orbit bijection. Thus the theorem is true for when \(G\) is a Euclidean group \(E(m,0)\) or \(E(0,n)\). For when \(H_p\cong E(m-1,n-1)\) we proceed by induction, assuming that the orbit bijection is already true for this group; our base cases being the groups \(E(m,0)\) and \(E(0,n)\), and \(E(1,1)\) which we verified in Example~\ref{poinc11}. 
\end{proof}
\subsection{An iterative method for obtaining orbit types}
Consider the set of orbits in \(\Delta=\Delta_{m,n}\) for \(G=E(m,n)\). Identify \(V\) with its dual using the form \(Q\) and recall that the orbits of \(V\) correspond to the sets: \(Q(v,v)\) equal to a non-zero constant, \(v=0\), and the set of non-zero null vectors. It follows that every orbit of \(\Delta_{m,n}\) contains a point of the form: \((\omega,0)\), \((\omega,t\tau)\), \((\omega,s\sigma)\), and \((\omega,\nu)\); for \(\tau\), \(\sigma\), and \(\nu\) fixed timelike, spacelike, and non-zero and null vectors respectively, and scalars \(t,s>0\). From Proposition~\ref{sopq_stabilisers}, it follows that the set of orbits through these points is in bijection with the set of adjoint orbits of \(\mathfrak{so}(m,n)\), \(\mathfrak{so}(m-1,n)\times\mathbb{R}^{>0}_t\), \(\mathfrak{so}(m,n-1)\times\mathbb{R}^{>0}_s\), and \(\mathfrak{se}(m-1,n-1)\) respectively. By Theorem~\ref{poinc_bij}, the set of adjoint orbits of \(\mathfrak{se}(m-1,n-1)\) is itself also in bijection with the set of  \(O(m-1,n-1)\)-orbits through \(\Delta_{m-1,n-1}\). We may apply the same argument iteratively to obtain a hierarchy of orbit types, as demonstrated in Figure~\ref{poinc_tree}.
\begin{thm}
	The set of orbits in \(\Delta\) for \(G=E(m,n)\) is in bijection with the set of adjoint orbits of \(O(m-k,n-k)\) for \(0\le k\le\min\{m,n\}\), \(O(m-k-1,n-k)\times\mathbb{R}^{>0}_t\) for \(0\le k\le\min\{m-1,n\}\), and \(O(m-k,n-k-1)\times\mathbb{R}^{>0}_s\) for \(0\le k\le\min\{m,n-1\}\).
\end{thm}
\begin{figure}
	\centering
	\begin{tikzcd}
		& \Delta_{m,n}\arrow[ld]\arrow[rd]\arrow[d]\arrow[rdd] &  &  \\
		\mathfrak{so}(m,n) & \mathfrak{so}(m-1,n)\times\mathbb{R}^{>0}_t & \mathfrak{so}(m,n-1)\times\mathbb{R}^{>0}_s &  \\
		&  & \Delta_{m-1,n-1} \arrow[ld,dash,dashed]\arrow[rd,dash,dashed]\arrow[d,dash,dashed]\arrow[rdd,dashed,dash]&  \\
		& ~& ~ & ~ \\
		&  &  & \Delta_{m-n,0}\arrow[ld]\arrow[d] \\
		&  & \mathfrak{so}(m-n)^* & \mathfrak{so}(m-n-1)^*\times\mathbb{R}^{>0}_t \\
	\end{tikzcd}
\caption{\label{poinc_tree}Hierarchy of orbit types for \(E(m,n)\) with \(m>n\).}
\end{figure}
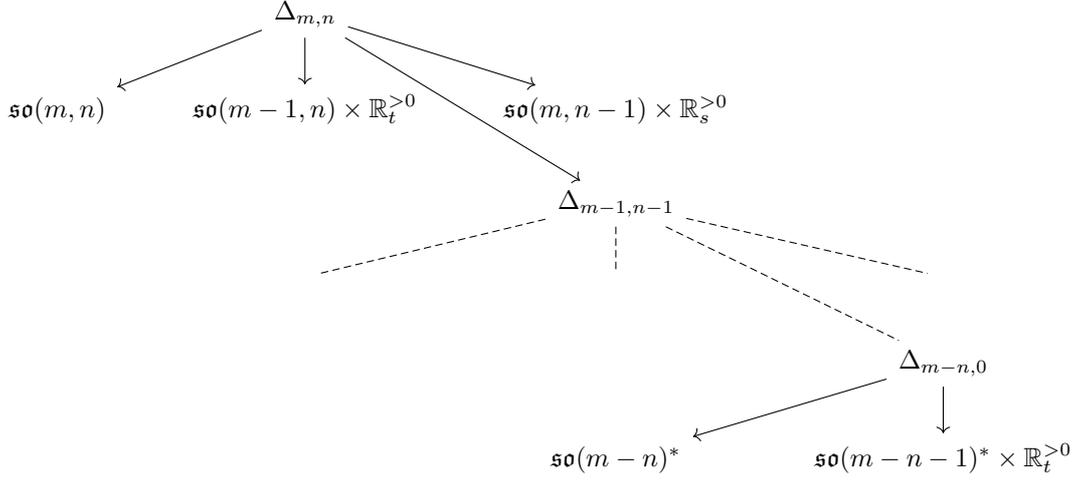
\begin{ex}[Orbit types for $E(1,3)$]
	The set of orbits in \(\Delta_{1,3}\) for \(G=E(1,3)\) is in bijection with the set of adjoint orbits of \(O(1,3)\), \(O(0,3)\times\mathbb{R}^{>0}_t\), \(O(1,2)\times\mathbb{R}^{>0}_s\), \(O(0,2)\), and \(O(0,1)\times\mathbb{R}^{>0}_s\). We may find explicit normal forms for these orbits.
	
	\begin{enumerate}
		\item[\textendash] The group \(O(1,3)\) is isomorphic to \(SL(2;\mathbb{C})\) viewed as a real Lie group. The adjoint orbits of this group are equal to the sets of \(2\times 2\) traceless, complex matrices \(\xi\) with a fixed non-zero determinant \(\zeta=\det\xi\), along with the origin \(\xi=0\), and the nilpotent orbit through the nilpotent Jordan block \(\xi=N_2\).
		
		\item[\textendash] The adjoint orbits of \(O(3)\) are spheres parametrised by their radius \(\rho\ge 0\).
		
		\item[\textendash] The adjoint representation of \(O(1,2)\) is isomorphic to the vector representation on \(\mathbb{R}^{1+2}\). The orbits are equal to the sets of \(v\in\mathbb{R}^{1+2}\) with \(Q(v,v)=c\) for \(c\ne0\), along with the origin \(v=0\), and the set of non-zero null vectors.
		
		\item[\textendash] The adjoint orbit of \(O(2)\) through \(x\in\mathfrak{so}(2)\cong\mathbb{R}\) is equal to the set \(\{x,-x\}\). The orbits are therefore parametrised by \(x\ge 0\).
		
		\item[\textendash] The group \(O(1)\) has a single point orbit.
	\end{enumerate}
	We enumerate a list of orbit types for \(E(1,3)\) demonstrated in Figure~\ref{e13_forms}. We count fourteen orbit types collected into five groups. This coincides with that given in \cite{cush06} in Tables 3 and 6.
	\begin{figure}[h]
		\[
		\renewcommand\arraystretch{1.1}
		\begin{array}{l|l}
		&\text{Adjoint orbit normal forms} \\ \hline O(1,3) & \zeta=x+iy=\det\xi\in\mathbb{C}\setminus\{0\} \\
		& \zeta=0 \\
		& \xi= N_2 \\ \hline
		O(3)\times\mathbb{R}^{>0}_t & (t,\rho),~t>0,~\rho>0 \\
		& (t,0),~t>0,~\rho=0 \\ \hline
		O(1,2)\times\mathbb{R}^{>0}_s & (s,c),~s>0,~c=Q(v,v)\in\mathbb{R}\setminus\{0\} \\
		& (s,0),~s>0,~c=0 \\
		& s>0,\text{ for $v$ a non-zero null vector in $\mathbb{R}^{1+2}$} \\ \hline
		O(2) & x>0 \\ 
		& x=0 \\ \hline 
		O(1)\times\mathbb{R}^{>0}_s & s>0
		\end{array}
		\]
		\caption{\label{e13_forms}Orbit types for \(E(1,3)\).}
	\end{figure}
\end{ex}

\section{A homotopy equivalence between orbits} 
\subsection{Showing bijected orbits are homotopy equivalent}
Consider two bijected orbits \(\orb\subset\mathfrak{g}\) and \(\orb^*\subset\mathfrak{g}^*\). Recall that they are both affine bundles over their corresponding orbits \(X\subset\Sigma\) and \(Y\subset\Pi\) respectively and hence each share the same homotopy type as them. In Theorem~\ref{orb_bijection} the orbit bijection is established using an intermediate orbit \(Z\) in \(\Delta\) to which both \(X\) and \(Y\) correspond. Our strategy will be to show that both \(X\) and \(Y\) are homotopy equivalent to \(Z\) and therefore so too are \(\orb\) and \(\orb^*\).

 For a group \(H\) suppose we have a finite collection of \(H\)-spaces together with \(H\)-equivariant bundle maps connecting them in the sense below.

\begin{equation}\label{zigzag_property}
\begin{tikzcd}
E_1\arrow[rd] &  & E_2\arrow[ld]\arrow[rd,dash,dashed] &  & E_{n+1}\arrow[ld]& \\~  & F_1 &  & F_n &  
\end{tikzcd}
\end{equation}
If the fibres of all these bundles are contractible then we will say that any two of these spaces are \emph{zigzag related}. There are now two things to note: that being zigzag-related is an equivalence relation on \(H\)-spaces; and, that being zigzag related also means that the two spaces have the same homotopy type.

Now let \(B\) be a \(G\)-space and for some \(b\in B\) let \(H\) denote the isotropy subgroup \(G_b\). The map \(G\rightarrow B\) given by sending \(g\) to \(gb\) defines a principal \(H\)-bundle over \(B\). Now suppose that \(E\) is some \(H\)-space and recall the definition of the associated fibre bundle \(B_E\coloneqq(G\times E)/H\); a bundle over \(B\) with fibre \(E\). This is given by the group quotient of \(G\times E\) with respect to the \(H\)-action \(h(g,x)=(gh,h^{-1}x)\). The space \(B_E\) also inherits a transitive \(G\)-action given by \(\tilde{g}[(g,x)]=[(\tilde{g}g,x)]\) which commutes with the bundle projection \(B_E\rightarrow B\).
\begin{propn}\label{bundle_lemma}
	Let \(X\) be a \(G\)-space together with a \(G\)-equivariant bundle map \(X\rightarrow B\) with fibre \(F\) above \(b\). Suppose there is an \(H\)-equivariant bundle map \(\phi\colon E\longrightarrow F\) with fibre \(D\). Then there is a \(G\)-equivariant bundle map \(B_E\rightarrow X\) with fibre \(D\) such that the following diagram commutes:
	\begin{equation}
	\begin{tikzcd}
	& B_E\arrow[d]\arrow[ld] \\ X \arrow[r] & B
	\end{tikzcd}
	\end{equation}
\end{propn}
\begin{proof}
	Fix some \(x_0\in E\) and observe that any point in \(B_E\) may be represented by a class of the form \([(g,x_0)]\). The bundle map in question is given by sending \([(g,x_0)]\) to \(g\phi(x_0)\). This is readily seen to be well defined and \(G\)-equivariant.
\end{proof}
\begin{cor}\label{unique_bundle}
	Any \(G\)-space \(X\) with equivariant fibre bundle \(X\longrightarrow B\) and fibres \(H\)-equivariantly diffeomorphic to \(E\) is itself \(G\)-equivariantly diffeomorphic to the associated bundle \(B_E\).
\end{cor}
\begin{lem}[Zigzag Lemma]\label{zigzag_lemma}
	Consider a collection of \(H\)-spaces as in equation~\eqref{zigzag_property} which are zigzag related. Then the corresponding associated fibre bundles over \(B\) are also zigzag related.
\end{lem}
\begin{proof}
	A direct application of Proposition~\ref{bundle_lemma} shows that we may lift the zigzag of bundle maps in equation~\eqref{zigzag_property} to
	\begin{equation}
	\begin{tikzcd}
	B_{E_1}\arrow[rd] &  & B_{E_2}\arrow[ld]\arrow[rd,dash,dashed] &  & B_{E_{n+1}}\arrow[ld]& \\~  & B_{F_1} &  & B_{F_n} &  
	\end{tikzcd}
	\end{equation}
	whose fibres are all contractible.
\end{proof}
\begin{thm}[Homotopy-type preserving bijection]\label{homotopic_bij}
	In addition to the hypotheses in Theorem~\ref{orb_bijection}, suppose further that the bijected \(H_\omega\)- and \(H_p\)-orbits are zigzag related. Then the adjoint and coadjoint orbits of \(G\) which are in bijection with each other are also zigzag related; in particular, they are homotopy equivalent.
\end{thm}
\begin{proof}
	Let \(Z\) be an orbit in \(\Delta\) through some \((\omega,p)\), and \(X\) and \(Y\) the corresponding orbits in \(\Sigma\) and \(\Pi\) which are both in bijection with \(Z\). Both \(X\) and \(Z\) are \(H\)-equivariant bundles over the adjoint orbit \(H\cdot\omega\) through \(\omega\in\mathfrak{h}\) whose fibres are respectively given by bijected orbits in \(\ker\omega^*\) and its dual. By Corollary~\ref{unique_bundle}, both \(X\) and \(Y\) are associated fibre bundles to the principal bundle \(H\rightarrow H\cdot\omega\). As the fibres are assumed to be zigzag related, it follows from the Zigzag Lemma that \(X\) and \(Z\) are also zigzag related; let's write this as \(X\sim Z\). A similar argument with \(Y\) also establishes that \(Y\sim Z\). Now let \(\orb\) and \(\orb^*\) denote the adjoint and coadjoint orbits corresponding to \(X\) and \(Y\) respectively. As we have shown, \(\orb\) is an equivariant bundle over \(X\) with affine fibre \(\Imag\omega\), and \(\orb^*\) an equivariant bundle over \(Y\) with affine fibre \(\mathfrak{h}_p^\circ\). Therefore \(\orb\sim X\), \(\orb^*\sim Y\) and thus \(\orb\sim\orb^*\).
\end{proof}
\subsection{The case for the Poincar\'{e} group}

\begin{propn}
	For \(H=O(m,n)\) and some \(\omega\in\mathfrak{so}(n,m)\), the \(H_\omega\)-orbit bijection given in Proposition~\ref{ortho_flag_bij} has the property that two orbits in bijection with each other are zigzag related.
\end{propn}
\begin{proof}
	For when the orbit in question is the origin, this is clear. Consider then the orbit through a non-zero \(p\in\ker\omega\) contained to a set \(E_j\setminus E_{j+1}\) with respect to the flag given in \eqref{distinct_flag} for \(0\le j\le k\), and the corresponding bijected orbit through \(\varphi(p)\in E_{j+1}^\circ\setminus E_j^\circ\). There is an equivariant bundle map from the orbit through \(p\) to the orbit of \(O(E_j/E_{j+1};Q_j)\) through \([p]\) whose fibres are translates of \(E_{j+1}\); thus the fibres are contractible. The orbit through \(\varphi(p)\) is likewise a bundle over the \(O(E_{j+1}^\circ/E_j^\circ;Q^*_j)\)-orbit through \([\varphi(p)]\) with contractible fibres equal to translates of \(E_j^\circ\). Since the group \(H_\omega\) preserves the form \(Q_j\), the isomorphism \(\varphi\colon E_{j}/E_{j+1}\rightarrow E^\circ_{j+1}/E^\circ_{j}\) is equivariant with respect to \(H_\omega\), and therefore the orbits through \([p]\) and \([\varphi(p)]\) are \(H_\omega\)-equivariantly diffeomorphic via the map \(\varphi\). Thus the orbits through \(p\) and \(\varphi(p)\) are zigzag related.
\end{proof}

\begin{thm}
	For \(G=E(m,n)\), consider the orbit bijection given in Theorem~\ref{poinc_bij}. Take an adjoint and coadjoint orbit both in bijection with each other (via a bijected orbit in \(\Delta\)). These two orbits are zigzag related; in particular, bijected adjoint and coadjoint orbits are homotopy equivalent.
\end{thm}
\begin{proof}
	To apply Theorem~\ref{homotopic_bij} we need to show that bijected \(H_p\)- and \(H_\omega\)-orbits are zig-zag related. The proposition above demonstrates that this is true for the centralizer group orbits. It remains to show that it is true for the orbits of \(H_p\). From Proposition~\ref{sopq_stabilisers}, these groups are isomorphic to \(O(m,n)\), \(O(m-1,n)\), \(O(m,n-1)\) and \(E(m-1,n-1)\) for when \(p\) is zero, timelike, spacelike, and non-zero and null respectively (and for whenever the entries are non-negative). For the first three cases, these groups are semisimple, and thus the adjoint and coadjoint representations are isomorphic; consequently the trivial orbit bijection is equivariant and bijected orbits are zigzag related. Thus the theorem is true for when \(G\) is a Euclidean group \(E(m,0)\) or \(E(0,n)\). For when \(H_p\cong E(m-1,n-1)\) we proceed by induction, assuming that bijected orbits are zigzag related; our base cases being the groups \(E(m,0)\) and \(E(0,n)\), and \(E(1,1)\) which is verified from Example~\ref{poinc11}. \end{proof}
	\subsection{The case for the affine group}
	Frustratingly, Theorem~\ref{homotopic_bij} cannot be directly applied to the affine group without some modification. We will here explain the problem and briefly sketch its resolution. It can indeed be shown that bijected orbits are homotopic, however we shall be consciously light on the details, leaving a rigorous proof as an exercise for the interested reader.
	
	The attempted proof proceeds analogously to the case of the Poincar\'{e} group. Consider the \(H_\omega\)-orbit \(E_j\setminus E_{j+1}\) together with the corresponding orbit \(E_{j+1}^\circ\setminus E_j^\circ\) as given in \eqref{flag_bij}. Each of these orbits is an equivariant bundle with contractible fibres over the non-zero vector orbits of \(GL(E_j/E_{j+1})\) and \(GL(E_{j+1}^\circ/ E_j^\circ)\) respectively. The problem now lies with the fact that, with respect to the canonical isomorphism \(E_{j+1}^\circ/ E_j^\circ\cong(E_j/E_{j+1})^*\), these two orbits, although identical, are not equivariantly isomorphic. In particular, the bijected \(H_\omega\)-orbits are not in general zig-zag related.
	
	A remedy to this problem is to define a notion of being `pseudo-equivariant', whereby a map \(\varphi\) satisfies \(\varphi(rp)=r^{-T}\varphi(p)\) for all \(r\in GL(n)\). One then weakens the definition of being zigzag related in \eqref{zigzag_property} to allow pseudo-equivariant bundle maps between the spaces. After proving the zigzag lemma for this new weakened definition, the proof of Theorem~\ref{homotopic_bij} follows verbatim and may be applied to the affine group.

\section*{Conclusions}

To what extent we have found a geometric explanation for the orbit bijection found in \cite{cush06} is debatable. Although we have demonstrated a framework and strategy for proving such a result, the problem has now shifted into a similar bijection question concerning the little-groups and centralizer subgroups. 

In practice, a key step in establishing the orbit bijections for the affine and Poincar\'{e} groups was an induction argument which took advantage of the fact that the subgroups \(H_p\) were either reductive, or equal to an affine or Poincar\'{e} group defined on a space of lower dimension. Using the same idea, it is possible to prove the orbit bijection for other semidirect products, including the special and connected versions of the affine and Poincar\'{e} groups, and even the Galilean group. It is worth noting the limitations however of this inductive argument. Consider for example the semidirect product \[\text{Symp}(2n;\mathbb{R})\ltimes\mathbb{R}^{2n}\] of the symplectic group with its defining vector representation. The little subgroup \(\text{Symp}(2n;\mathbb{R})_p\) fixing a non-zero \(p\) is called the \emph{odd symplectic group}, and for \(n>1\) is isomorphic to the semidirect product \(\text{Symp}(2n-2;\mathbb{R})\ltimes H_{2n-2}\) \cite{odd_symp}. Here \(H_{2n-2}\) is the Heisenberg group corresponding to the symplectic vector space \(\mathbb{R}^{2n-2}\). For this example, our inductive argument no longer applies. However, fortunately the bijection result may still be rescued by realising that this odd symplectic group (which may be found in the literature alternatively by the names affine extended symplectic group or Schr\"{o}dinger group) is a one-dimensional central extension of the original semidirect product defined for a smaller dimension. As central elements are unaffected by the adjoint representation, the bijection result still holds for this example.

The obvious question is to ask: to what extent does such an orbit bijection result hold for other groups? This author, at the time of writing, has not encountered a single example of a group which does not exhibit a geometric bijection between the sets of adjoint and coadjoint orbits, together with the property that bijected orbits share the same homotopy type. It is tempting then to conjecture that perhaps this result is true, if not for all groups, but for a large class of groups. 

A next step could be to consider the general case of a Lie algebra \(\mathfrak{g}\) containing some ideal. In \cite{myk}, they generalise the bundle-of-little-group-orbits construction given in \cite{rawnsley} to any \(\mathfrak{g}\), and obtain an analogous classification of the coadjoint orbits of \(\mathfrak{g}\) with respect to a given ideal. It might then be possible to expand on this work, and derive a set analogous to our set \(\Delta\); the set of orbits through which might be shown to be in bijection with each of the sets of adjoint and coadjoint orbits of \(\mathfrak{g}\).
\bibliographystyle{alpha}
\bibliography{UpdatedPaper2} 

\newcommand{\Addresses}{{
		\bigskip
		\footnotesize
		
		P.~Arathoon, \textsc{School of Mathematics, University of Manchester,
			Manchester, M13 9PL, U.K.}\par\nopagebreak
		\textit{E-mail address},  \texttt{philip.arathoon@manchester.ac.uk}
		
%
%
%
		
}}

\Addresses

\end{document}